\documentclass[final]{siamart171218}

\usepackage{amssymb}
\usepackage{mathdots}
\usepackage{blkarray}
\usepackage{multirow}
\usepackage{caption}
\usepackage{subcaption}

\newcommand{\AAA}{A_n}
\newcommand{\ar}{A_R}

\newcommand{\af}{A_{M}}
\newcommand{\A}{A_n}
\newcommand{\Aml}{A_{n}}
\newcommand{\arml}{A_R}

\newcommand{\Ce}{C_n}
\newcommand{\CS}{G_n}
\newcommand{\CSt}{\Ce^{(S)}}

\newcommand{\En}{E_n}
\newcommand{\Enh}{{\widehat{E}_n}}
\newcommand{\Eml}{E_{n}}
\newcommand{\fvar}{\theta}
\newcommand{\ff}{f}
\newcommand{\ffr}{f_R}
\newcommand{\ffi}{f_I}
\newcommand{\Hml}{H_{n}}
\newcommand{\Yml}{Y_{n}}
\newcommand{\Ha}{H_n}
\newcommand{\Pre}{P_n}
\newcommand{\Y}{Y_n}
\newcommand{\D}{\mathcal{D}}
\newcommand{\YYY}{Y_n}

\newcommand{\ee}{\mathrm{e}}
\newcommand{\dd}{\mathrm{d}}
\newcommand{\ii}{\mathrm{i}}
\newcommand{\Rnn}{\mathbb{R}^{n\times n}}
\newcommand{\Rn}{\mathbb{R}^{n}}

\newcommand{\Lampos}{\Lambda_{\mathrm{pos}}}

\newcommand{\Unp}{U_{\text{sym}}}
\newcommand{\Unn}{U_{\text{skew}}}

\DeclareMathOperator*{\essinf}{essinf}
\DeclareMathOperator*{\esssup}{esssup}

\DeclareMathOperator{\diag}{diag}

\newsiamremark{remark}{Remark}
\newsiamremark{example}{Example}

\title{Preconditioners for symmetrized Toeplitz and multilevel Toeplitz matrices%
\thanks{Version of January 10, 2019. This research was supported by Engineering and Physical Sciences Research Council First Grant EP/R009821/1.}}
\author{J. Pestana%
\thanks{Department of Mathematics and Statistics,
         University of Strathclyde,
         26 Richmond Street,
         Glasgow, G1 1XH, UK
         (jennifer.pestana@strath.ac.uk).}}

\ifpdf
\hypersetup{
  pdftitle={Preconditioners for symmetrized Toeplitz and multilevel Toeplitz matrices},
  pdfauthor={J. Pestana}
}
\fi

\begin{document}
\maketitle

\begin{abstract}
When solving linear systems with nonsymmetric Toeplitz or multilevel Toeplitz matrices using Krylov subspace methods, the coefficient matrix may be symmetrized. The preconditioned MINRES method can then be applied to this symmetrized system, which allows rigorous upper bounds on the number of MINRES iterations to be obtained. 
However, effective preconditioners for symmetrized (multilevel) Toeplitz matrices are lacking. Here, we propose novel ideal preconditioners, and investigate the spectra of the preconditioned matrices. We show how these preconditioners can be approximated and demonstrate their effectiveness via numerical experiments. 
\end{abstract}



\section{Introduction}
Linear systems 
\begin{equation}
\label{eqn:linsys}
\AAA x = b, 
\end{equation}
where $\AAA\in\mathbb{R}^{n\times n}$ is a Toeplitz or multilevel Toeplitz matrix, and $b\in\mathbb{R}^n$ arise in a range of applications. These include the discretization of partial differential and integral equations, time series analysis, and signal and image processing~\cite{chan2007,ng2004}. 
Additionally, demand for fast numerical methods for fractional diffusion problems---which have recently received significant attention---has renewed interest in the solution of Toeplitz and Toeplitz-like systems~\cite{DMS16,MDDM17,PKNS14,PaSu12,WWS10}. 

Preconditioned iterative methods are often used to solve systems of the form~\cref{eqn:linsys}. When $\AAA$ is Hermitian, CG~\cite{hestenes1952} and MINRES~\cite{paige1975} can be applied, and their descriptive convergence rate bounds guide the construction of effective preconditioners~\cite{chan2007,ng2004}. On the other hand, convergence rates of preconditioned iterative methods for nonsymmetric Toeplitz matrices are difficult to describe. Consequently, preconditioners for nonsymmetric problems are typically motivated by heuristics. 

As described in \cite{PeWa15} for Toeplitz matrices, and discussed in \cref{sec:ml_background} for the multilevel case, $\AAA$ is symmetrized by the exchange matrix 
\begin{equation}
\label{eqn:Y}
\YYY =
\begin{bmatrix} & & 1\\ & \iddots & \\ 1 & & \end{bmatrix}, 
\end{equation}
so that \cref{eqn:linsys} can be replaced by 
\begin{equation}
\label{eqn:symlinsys}
\YYY\AAA x = \YYY b,
\end{equation}
with the symmetric coefficient matrix $\YYY\AAA$. 
Although we can view $\YYY$ as a preconditioner, its role is not to accelerate convergence, and there is no guarantee that \cref{eqn:symlinsys} is easier to solve than \cref{eqn:linsys} \cite{FFHAS18,MaPe19}. Instead, the presence of $\YYY$ allows us to use preconditioned MINRES, with its nice properties and convergence rate bounds, to solve \cref{eqn:symlinsys}.  
We can then apply a secondary preconditioner $\Pre\in\Rnn$ to improve the spectral properties of $\YYY\AAA$, and therefore accelerate convergence. An additional benefit is that MINRES may be faster than GMRES~\cite{SaSc86} even when iteration numbers are comparable, since it requires only short-term recurrences. 

Preconditioned MINRES  requires a symmetric positive definite preconditioner $\Pre$, but it is not immediately clear how to choose this matrix when $\AAA$ is nonsymmetric. In~\cite{PeWa15} it was shown that absolute value circulant preconditioners, which we describe in the next section, give fast convergence for many Toeplitz problems. However, for some problems there may be more effective alternatives based on Toeplitz matrices (see, e.g.~\cite{ChNg93,HaNa94}). Moreover, multilevel circulant preconditioners are not generally effective for multilevel Toeplitz problems~\cite{Serr02,SeTy99,SeTy03}. 
Thus, alternative preconditioners for \cref{eqn:symlinsys} are needed.

In this paper, we describe ideal preconditioners for symmetrized (multilevel) Toeplitz matrices and show how these can be effectively approximated. To set the scene, we present background material in \cref{sec:background}. \Cref{sec:idealpre,sec:idealpreml} describe the ideal preconditioners for Toeplitz and multilevel Toeplitz problems, respectively. Numerical experiments in \cref{sec:numerics} verify our results and show how the ideal preconditioners can be efficiently approximated by circulant matrices or multilevel methods. Our conclusions can be found in \cref{sec:conc}.

\section{Background}
\label{sec:background}
In this section we collect pertinent results on Toeplitz and multilevel Toeplitz matrices. 

\subsection{Toeplitz and Hankel matrices}
\label{sec:backgroundtoep}
Let $\A\in\Rnn$ be the nonsingular Toeplitz matrix
\begin{equation}
\label{eqn:toeplitz}
\A =
\begin{bmatrix}
a_0 & a_{-1} & \dots & a_{-n+2} & a_{-n+1}\\
a_1 & a_0 & a_{-1} &  & a_{-n+2}\\
\vdots & a_1 & a_0 & \ddots & \vdots\\
a_{n-2} &  & \ddots & \ddots & a_{-1}\\
a_{n-1} & a_{n-2} & \dots & a_1 & a_0
\end{bmatrix}.
\end{equation}

In many applications, the matrix $\A$ is associated with a generating function $f \in L^1([-\pi,\pi])$ via its Fourier coefficients 
\begin{equation}
\label{eqn:fourier}
a_k = \frac{1}{2\pi}\int_{-\pi}^\pi f(\fvar)\ee^{-\ii k\theta} \dd \fvar, \quad k \in \mathbb{Z}.
\end{equation}
We use the notation $\A(f)$ when we wish to stress that a Toeplitz matrix $\A$ is associated with the generating function $f$. 
An important class of generating functions is the Wiener class, which is the set of functions satisfying 
\[f(\fvar) = \sum_{k=-\infty}^\infty a_k\ee^{-\ii k\theta}, \quad \sum_{k=-\infty}^\infty |a_k| < \infty.\]

Many properties of $\A(f)$ can be determined from $f$. For example, if $f$ is real than $\A(f)$ is Hermitian and  
its eigenvalues are characterized by $f$ \cite[pp. 64--65]{GrSe01}. On the other hand, if $f$ is complex-valued then $\A(f)$ is non-Hermitian for at least some $n$ and its singular values are characterized by $|f|$~\cite{Avra88, Part86, Wido89}. 

Circulant matrices are Toeplitz matrices of the form 
\[
\Ce =
\begin{bmatrix}
c_0 & c_{n-1} & \dots & c_{2} & c_{1}\\
c_1 & c_0 & c_{n-1} &  & c_{2}\\
\vdots & c_1 & c_0 & \ddots & \vdots\\
c_{n-2} &  & \ddots & \ddots & c_{n-1}\\
c_{n-1} & c_{n-2} & \dots & c_1 & c_0
\end{bmatrix}.
\]
They are diagonalized by the Fourier matrix, i.e., $\Ce = F_n^\ast \Lambda_n F_n$, where 
\[(F_n)_{j,k} = \frac{1}{\sqrt{n}}\ee^{\frac{2\pi\ii j k}{n}},\quad j,k = 0,\dotsc, n-1,\]
 and $\Lambda_n = \diag(\lambda_0,\dotsc,\lambda_{n-1})$, with 
\begin{equation}
\label{eqn:circeigs}
\lambda_k = \sum_{j=0}^{n-1}c_j\ee^{\frac{2\pi \ii j k}{n}}.
\end{equation}
We denote by $\Ce(f)$ the circulant with eigenvalues $\lambda_j = f(2\pi j/n)$, $j = 0,\dotsc, n-1$. 
The absolute value circulant~\cite{CNP96,NgPo01,PeWa15} derived from a circulant $\Ce$ is the matrix 
\begin{equation}
\label{eqn:abscirc}
|\Ce| = F_n^\ast |\Lambda_n| F_n.
\end{equation}

Closely related to Toeplitz matrices are Hankel matrices $\Ha\in\mathbb{R}^{n\times n}$,
\[
\Ha =
\begin{bmatrix}
a_1 & a_{2} & a_3 & \dots & a_{n}\\
a_2 & a_3 &  & \iddots & a_{n+1}\\
a_3 &  & \iddots& \iddots & \vdots\\
\vdots & \iddots &  \iddots&  & a_{2n-2}\\
a_{n} & a_{n+1} & \dots & a_{2n-2} & a_{2n-1}
\end{bmatrix},
\]
which have constant anti-diagonals. 
It is well known that a Toeplitz matrix can be converted to a Hankel matrix, or vice versa, by flipping the rows (or columns), i.e., via $\Y$ in \cref{eqn:Y}. 
Since Hankel matrices are necessarily symmetric, this means that any nonsymmetric Toeplitz matrix $\A$ can be symmetrized by applying $\Y$, so that 
\begin{equation}
\label{eqn:selfadjoint}
\Y\A = \A^T\Y.
\end{equation}
Alternatively, we may think of $\A$ being self-adjoint with respect to the bilinear form induced by $\Y$~\cite{gohberg2005,pestana2011a}. 
 
A matrix $\CS\in\Rnn$ is centrosymmetric  if 
\begin{equation}
\label{eqn:centro}
\CS\Y=\Y\CS
\end{equation}
 and is skew-centrosymmetric if $\CS\Y = -\Y\CS$. 
Thus,  \cref{eqn:selfadjoint} shows that symmetric Toeplitz matrices are centrosymmetric. It is clear from \cref{eqn:centro} that the inverse of a nonsingular centrosymmetric matrix is again centrosymmetric. Furthermore, nonsingular centrosymmetric matrices have a centrosymmetric square root~\cite[Corollary 1]{LZR07}\footnote{In~\cite{LZR07} the proof is given only for a centrosymmetric matrix of even dimension. However, the extension to matrices of odd dimension is straightforward.}.

\subsection{Multilevel Toeplitz and Hankel matrices} \label{sec:ml_background}
Multilevel Toeplitz matrices are generalizations of Toeplitz matrices. To define a generating function for a multilevel Toeplitz matrix, let $j = (j_1,\dotsc, j_p) \in \mathbb{Z}^p$ be a multi-index and consider a $p$-variate function $f \in L^1([-\pi,\pi]^p)$, $f:[-\pi,\pi]^p \rightarrow \mathbb{C}$. The Fourier coefficients of $f$ are defined as  
\begin{equation*}
a_j = a_{(j_1,\dotsc,j_p)}= \frac{1}{(2\pi)^p} \int_{[-\pi,\pi]^p}f(\fvar)\ee^{-\ii\langle \fvar,j\rangle}\,\dd \fvar, \qquad j \in \mathbb{Z}^p,
\end{equation*}
where  $\langle \fvar,j\rangle = \sum_{k=1}^p \fvar_kj_k$ and $\dd \fvar = \dd \fvar_1\dots\dd \fvar_p$  is the volume element with respect to the $p$-dimensional Lebesgue measure.

If $n = (n_1,\dotsc, n_p)\in\mathbb{N}^p$, with $n_i>1$, $i = 1,\dotsc, p$, and $\pi(n) = n_1\dots n_p$, then $f$ is the generating function of the multilevel Toeplitz matrix $\Aml(f)\in\mathbb{R}^{\pi(n)\times \pi(n)}$, where
\[
\Aml(f) = \sum_{j_1 = -n_1+1}^{n_1-1}\dots\sum_{j_p = -n_p+1}^{n_p-1}J_{n_1}^{(j_1)}\otimes \dots \otimes J_{n_p}^{(j_p)} {a}_{(j_1,\dotsc,j_p)}.
\]
Here, $J^{(k)}_r \in \mathbb{R}^{r\times r}$ is the matrix whose $(i,j)$th entry is one if $i-j=k$ and is zero otherwise. 

Similarly, we can define a multilevel Hankel matrix as 
\[\Hml(f) = \sum_{j_1 = 1}^{2n_1-1}\dots\sum_{j_p = 1}^{2n_p-1}K_{n_1}^{(j_1)}\otimes \dots \otimes K_{n_p}^{(j_p)} {a}_{(j_1,\dotsc,j_p)},\]
where $K^{(k)}_r \in \mathbb{R}^{r\times r}$ is the matrix whose $(i,j)$th entry is one if $i+j=k+1$ and is zero otherwise. Although a multilevel Hankel matrix does not necessarily have constant anti-diagonals, it is symmetric. 

Multilevel Toeplitz matrices can also be symmetrized by the exchange matrix $\Yml \in \mathbb{R}^{\pi(n) \times \pi(n)}$, $\Yml = Y_{n_1}\otimes \dots \otimes Y_{n_p}$.  To see this we use an approach analogous to that in the proof of~\cite[Lemma 5]{FaTi00}. The key point is that 
$Y_r J_r^{(k)} = K_r^{(r-k)}$, so that 
 \begin{align*}
\Yml \Aml(f) 
=& \sum_{j_1 = -n_1+1}^{n_1-1}\dots\sum_{j_p = -n_p+1}^{n_p-1}\left((Y_{n_1}J_{n_1}^{(j_1)})\otimes \dots \otimes (Y_{n_p}J_{n_p}^{(j_p)})\right) {a}_{(j_1,\dotsc,j_p)}\\
=& \sum_{j_1 = -n_1+1}^{n_1-1}\dots\sum_{j_p = -n_p+1}^{n_p-1}\left( K_{n_1}^{(n_1-j_1)}\otimes \dots \otimes K_{n_p}^{(n_p-j_p)}\right){a}_{(j_1,\dotsc,j_p)})\\
=& \sum_{j_1 = 1}^{2n_1-1}\dots\sum_{j_p = 1}^{2n_p-1}K_{n_1}^{(j_1)}\otimes \dots \otimes K_{n_p}^{(j_p)} {{b}}_{(j_1,\dotsc,j_p)},
\end{align*}
where  $b_{(j_1,\dotsc,j_p)} = a_{(n_1-j_1,\dotsc,n_p-j_p)}$. Thus, $\Yml\Aml(f)$  is a multilevel Hankel matrix, and hence is symmetric. 

\subsection{Assumptions and notation}
Throughout, we assume that all Toeplitz or multilevel Toeplitz matrices $\A$  are real, and are associated with generating functions in $L^1([-\pi,\pi]^p)$. We denote the real and imaginary parts of $\ff$ by $\ffr$ and $\ffi$, respectively, so that $\ff = \ffr + \ii \ffi$. We  assume that the symmetric part of $\A$, given by $\ar = (\A + \A^T)/2$, is positive definite, which is equivalent to requiring that $\ffr$ is essentially positive. Similarly, we assume that $|f|\ge \delta > 0$ for some constant $\delta$, so that $\A(|f|)$ is positive definite with $\lambda_{\min}(\A(|f|)\ge \delta$. Moreover,  $\lambda_{\min}(\A(|f|)>\delta$ if $\esssup|f|>\delta=\essinf |f|$ (see \cref{lem:eigfunction}).

\section{Preconditioning Toeplitz matrices}
 \label{sec:idealpre}
In this section we introduce our ideal preconditioners for \cref{eqn:symlinsys} when $\A$ is a Toeplitz matrix, and analyse the spectrum of the preconditioned matrices. Although these preconditioners may be too expensive to apply exactly, they can be approximated by, e.g., a circulant matrix or multigrid solver.

\subsection{The preconditioner $\ar$}\label{sec:idealpretoep}
The first preconditioner we consider is the symmetric part of $\A$, namely $\ar = (\A + \A^T)/2$, which was previously used to precondition the nonsymmetric problem \cref{eqn:linsys} (see \cite{HST05}). When applied to the symmetrized system \cref{eqn:symlinsys}, spectral information can be used to bound  the convergence rate of preconditioned MINRES. Accordingly, in this section we  characterize the eigenvalues of $\ar^{-\frac{1}{2}}\Y\A\ar^{-\frac{1}{2}}$.

We begin by stating a powerful result that characterizes the spectra of preconditioned  Hermitian Toeplitz matrices in terms of generating functions.

\begin{lemma}[{\cite[Theorem 3.1]{Serr99}}]\label{lem:eigfunction}
Let $f,g \in L^1([-\pi,\pi])$ be real-valued functions with $g$ essentially positive. Let $\A(f)$ and $\A(g)$ be the Hermitian Toeplitz matrices with generating functions  $f$ and $g$, respectively. Then, 
$\A(g)$ is positive definite and the eigenvalues of  $\A^{-1}(g)\A(f)$ lie in $(r,R)$, where $r<R$ and 
\[r = \essinf_{x \in [-\pi,\pi]}\frac{f(\fvar)}{g(\fvar)}, \qquad R = \esssup_{\fvar \in [-\pi,\pi]}\frac{f(\fvar)}{g(\fvar)}.\]
If $r=R$ then $\A^{-1}(g)\A(f)=I_n$, the identity matrix of dimension $n$. 
\end{lemma}
\cref{lem:eigfunction} shows that in the Hermitian case we can bound the extreme eigenvalues of preconditioned Toeplitz matrices using scalar quantities. If bounds on the eigenvalues nearest the origin are also available it is possible to estimate the convergence rate of preconditioned MINRES applied to the Toeplitz system. Unfortunately, this  result does not carry over to nonsymmetric matrices, nor is an eigenvalue inclusion region alone sufficient to bound the convergence rate of a Krylov subspace method for nonsymmetric problems~\cite{APS98,GPS96}. However, in the following we show that by symmetrizing the Toeplitz matrix $\A$ we can obtain analogous results to \cref{lem:eigfunction}, even for nonsymmetric $\A$. 
As a first step, we quantify the perturbation of the (nonsymmetric) preconditioned matrix $\ar^{-\frac{1}{2}}\A\ar^{-\frac{1}{2}}$ from the identity. 

\begin{lemma}\label{lem:IEdecomp}
Let $\ff \in L^1([-\pi,\pi])$, and let $\ff = \ff_R + \ii \ff_I$, where $\ff_R$ and $\ff_I$ are real-valued functions with $\ffr$ essentially positive. Additionally, let $\A:=\A(\ff)\in\mathbb{R}^{n\times n}$ be the Toeplitz matrix associated with $\ff$. Then $\ar = \A(\ffr)  = (\A + \A^T)/2$ is symmetric positive definite and 
\[\ar^{-\frac{1}{2}}\A\ar^{-\frac{1}{2}} = I_n + \En,\] 
where 
\[\|\En\|_2 = \epsilon < \esssup_{\theta \in [-\pi,\pi]} \left|\frac{\ff_I(\theta)}{\ff_R(\theta)}\right|.\]
\end{lemma}
\begin{proof}
It is easily seen from \cref{eqn:fourier} that 
 $\A(\ff) = \A(\ffr) +  \ii \A(\ffi)$. Moreover, from \cref{lem:eigfunction} we also know that $\ar = \A(\ffr)$ is symmetric positive definite and $\A(\ffi)$ is Hermitian.  It follows that 
\[\ar^{-\frac{1}{2}}\A\ar^{-\frac{1}{2}}  = \ar^{-\frac{1}{2}}(\ar + \ii \A(\ffi))\ar^{-\frac{1}{2}}  = I_n + \En,\]
where $\En =\ii \Enh= \ii\ar^{-\frac{1}{2}}\A(\ffi)\ar^{-\frac{1}{2}}$. 

To bound $\epsilon:=\|\En\|_2 = \|\Enh\|_2$, note that since $\Enh$ is Hermitian, $\|\Enh\|_2$ is equal to the spectral radius of $\Enh$. Applying \cref{lem:eigfunction} thus gives that 
\[\epsilon < \esssup_{\theta \in [-\pi,\pi]} \left|\frac{\ff_I(\theta)}{\ff_R(\theta)}\right|,\]
which completes the proof. 
\end{proof}

The above result tells us that the nonsymmetric preconditioned matrix will be close to the identity when the skew-Hermitian part of $\A$ is small, as expected. Although this enables us to bound the singular values of $\ar^{-\frac{1}{2}}\A\ar^{-\frac{1}{2}}$,
these cannot be directly related to the convergence of e.g., GMRES. In contrast, the following result will enable us to characterize the convergence rate of MINRES applied to \cref{eqn:symlinsys}. 

\begin{lemma}
\label{thm:symdecomp}
Let $\ff \in L^1([-\pi,\pi])$, and let $\ff = \ff_R + \ii \ff_I$, where $\ff_R$ and $\ff_I$ are real-valued functions with $\ffr$ essentially positive. Additionally, let $\A:=\A(\ff)\in\mathbb{R}^{n\times n}$ be the Toeplitz matrix associated with $\ff$. Then the symmetric positive definite matrix $\ar = \A(\ffr)  = (\A + \A^T)/2$ is such that 
\begin{equation}
\label{eqn:symdecomp}
\ar^{-\frac{1}{2}}(\Y\A)\ar^{-\frac{1}{2}} = \Y + \Y \En,
\end{equation}
where $\Y$ is the exchange matrix in \cref{eqn:Y} and 
\[\|\Y \En\|_2 =\epsilon < \esssup_{\theta \in [-\pi,\pi]}\left| \frac{\ffi(\theta)}{\ffr(\theta)}\right|.\]
\end{lemma}
\begin{proof}
Since $\ar$ is a symmetric Toeplitz matrix, it is centrosymmetric. Hence, $\ar^{-\frac{1}{2}}$ is centrosymmetric (see \cref{eqn:centro} and \cite{LZR07}), so that $\Y\ar^{-\frac{1}{2}} = \ar^{-\frac{1}{2}}\Y$. Combining this  with \cref{lem:IEdecomp} shows that 
\[\ar^{-\frac{1}{2}}(\Y\A)\ar^{-\frac{1}{2}} = \Y(\ar^{-\frac{1}{2}}\A\ar^{-\frac{1}{2}}) = \Y(I_n+\En) = \Y + \Y \En. \] 
Since $\Y$ is orthogonal, \(\|\Y \En\|_2 = \|\En\|_2\), and the result follows from~\cref{lem:IEdecomp}. 
\end{proof}

Applying Weyl's theorem \cite[Theorem 4.3.1]{HoJo90} to \cref{eqn:symdecomp} shows that the eigenvalues of $\ar^{-\frac{1}{2}}(\Y\A)\ar^{-\frac{1}{2}}$ lie in $[-1-\epsilon,-1+\epsilon]\cup [1-\epsilon,1+\epsilon]$. However, as $\epsilon$ grows, eigenvalues could move close to the origin, and hamper MINRES convergence. We will show that this cannot happen in the following result. 
\begin{theorem}
\label{thm:symeigs}
Let $\ff \in L^1([-\pi,\pi])$, and let $\ff = \ff_R + \ii \ff_I$, where $\ff_R$ and $\ff_I$ are real-valued functions with $\ffr$ essentially positive. Additionally, let $\A:=\A(\ff)\in\mathbb{R}^{n\times n}$ be the Toeplitz matrix associated with $\ff$ and let $\ar = \A(\ffr)  = (\A + \A^T)/2$. Then, the eigenvalues of $\ar^{-\frac{1}{2}}(\Y\A)\ar^{-\frac{1}{2}} $ lie in $[-1-\epsilon,-1]\cup[1,1+\epsilon]$, where 
\begin{equation}\label{eqn:arscalar}\epsilon < \esssup_{\theta \in [-\pi,\pi]} \left|\frac{\ffi(\theta)}{\ffr(\theta)}\right|.\end{equation}
\end{theorem}

\begin{proof}
We know from \cref{thm:symdecomp} that 
\[\ar^{-\frac{1}{2}}(\Y\A)\ar^{-\frac{1}{2}} = \Y + \Y \En, \]
where $\|\Y\En\|_2 < \epsilon$ and $\Y$ has eigenvalues $\pm 1$. 
Thus, as discussed above, the eigenvalues of $\ar^{-\frac{1}{2}}(\Y\A)\ar^{-\frac{1}{2}} $ lie in $[-1-\epsilon,-1+\epsilon]\cup [1-\epsilon,1+\epsilon]$. Hence, all that remains is to improve the bounds on the eigenvalues nearest the origin. 
Our strategy for doing so will be to apply  successive similarity transformations to $\Y + \Y \En$; as a byproduct, we will characterize the eigenvalues of $\ar^{-\frac{1}{2}}(\Y\A)\ar^{-\frac{1}{2}}$ in terms of the eigenvalues of $\Y \En$. 

Before applying our first similarity transform, we  recall from the proofs of \cref{lem:IEdecomp,thm:symdecomp} that $\ii \A(\ffi)$ is skew-symmetric and Toeplitz, while $\ar^{-\frac{1}{2}}$ is symmetric and centrosymmetric. It follows that  
$\Y \En$ is symmetric but  skew-centrosymmetric.  
Skew-centrosymmetry implies that whenever $(\lambda, x)$, $\lambda\ne 0$ is an eigenpair of $\Y \En$, then so is $(-\lambda, \Y x)$~\cite{HBW90,Pres98,Tren04}.  
Additionally, any eigenvectors of $\Y \En$ corresponding to a zero eigenvalue can be expressed as a linear combination of vectors of the form $x \pm \Y x$, $x\in\mathbb{R}^n$~\cite[Theorem 17]{Tren04}. 
Therefore, $\Y\En$ has eigendecomposition $\Y \En = U_n \Lambda_n U_n^T$, 
where 
\begin{equation}
\label{eqn:Lambda}
\Lambda_n = 
\begin{blockarray}{cccc}
& m_1 & m_1& m_2\\
\begin{block}{c[c c c]}
m_1 & \Lampos & & \\
 m_1& & -\Lampos & & \\
m_2 & & & 0 \\
\end{block}
\end{blockarray}
\end{equation}
and 
\begin{equation}
\label{eqn:U}
U_n = 
\begin{blockarray}{cccc}
m_1 & m_1 & m_3 & m_4\\
\begin{block}{[c c c c]}
U_{\mathrm{pos}} & \Y U_{\mathrm{pos}} & \Unp + \Y \Unp & \Unn - \Y \Unn\\
\end{block}
\end{blockarray},
\end{equation}
where $n = 2m_1 + m_2$ and $m_2 = m_3 + m_4$. 
Since $\Y\En$ is symmetric, we may assume that $U_n$ is orthogonal. 
We can now apply the first similarity transform, namely 
\begin{equation}
\label{eqn:strans1}
 U_n^T(\Y+\Y E)U_n = U_n^T \Y U_n + \Lambda_n .
 \end{equation}

Using the orthogonality of the columns of $U_n$, it is straightforward to show that 
\[U_n^T \Y U_n = 
\begin{bmatrix}
  & I_{m_1} &  & \\
  I_{m_1}&   &  &  \\
  &   & I_{m_3}& \\
  &   &  & -I_{m_4}\\
\end{bmatrix}. 
\]
Thus, $U_n^T\Y U_n= Q_n\Gamma_n Q_n^T$, where 
\[
\Gamma_n =
\begin{bmatrix}
 \widehat\Gamma_{2m_1} & & \\
 & I_{m_3} & \\
 & & -I_{m_4}\\
\end{bmatrix},\quad \text{ and } \quad 
Q_n = 
\left[ 
\begin{array}{c|c c}
\multirow{3}{*}{$\widehat Q$} & & \\
 & I_{m_3} & \\
 & & -I_{m_4}\\
\end{array}
\right].
\]
Here, $\widehat\Gamma_{2m_1} = \diag(1,-1,\dotsc,1,-1)$ and 
the $k$th column of $\widehat Q\in\mathbb{R}^{n\times 2m_1}$ is given by 
\[\widehat q_k = \begin{cases}
\frac{1}{\sqrt{2}}(e_{k} + e_{m_1+k}), & k \text{ odd},\\
\frac{1}{\sqrt{2}}(e_{ k} - e_{m_1+k}), & k \text{ even},
\end{cases}\]
with $e_j\in\Rn$ the $j$th unit vector. 
Consequently, our second similarity transform gives  
\begin{equation}
\label{eqn:simtrans2}
Q_n^T U_n^T (\Y + \Y \En) U_n Q_n =  \Gamma_n + Q_n^T \Lambda_n Q_n,
\end{equation}
with 
\[Q_n^T \Lambda_n Q_n = 
\left[
\begin{array}{c|c c}
 \begin{matrix} \Sigma_1 & & & \\ & \Sigma_2 & & \\ & & \ddots & \\ & & & \Sigma_{m_1}\end{matrix} & & \\
\hline
 & 0 & \\
 & &0\\ 
\end{array}
\right],
\]
where if 
\(\Lampos = 
\diag(\lambda_1,\lambda_2,\dotsc,\lambda_{m_1})\) then 
\[\Sigma _k= \begin{bmatrix} 0 & \lambda_k \\ \lambda_k & 0\end{bmatrix}.\]
Hence, letting 
\[Z = \begin{bmatrix}1 & \\ & -1 \end{bmatrix}\]
we find that 
\[ \Gamma + Q_n^T \Lambda_n Q_n = 
\begin{bmatrix}
\begin{matrix}Z+ \Sigma_1 & & & \\ &Z+ \Sigma_2 & & \\ & & \ddots & \\ & & & Z+\Sigma_{m_1}\end{matrix} & & \\
 & I_{m_3} & \\
 & &-I_{m_3}\\ 
\end{bmatrix}.
\]
Since the eigenvalues of $Z+ \Sigma_k$ are $\pm\sqrt{1+\lambda_k^2}$ we see from \cref{eqn:simtrans2} that the eigenvalues of $\ar^{-\frac{1}{2}}(\Y\A)\ar^{-\frac{1}{2}} $ are  $\pm\sqrt{1+\lambda_k^2}$, $k = 1, \dotsc, m_1$, and possibly one or both of  $1$ and $-1$. Hence, the eigenvalues are at least 1 in magnitude. This completes the proof. 
\end{proof}

\cref{thm:symeigs} characterizes the eigenvalues of $\ar^{-\frac{1}{2}}(\Y\A)\ar^{-\frac{1}{2}} $, and hence the convergence rate of preconditioned MINRES, in terms of the scalar quantity in \cref{eqn:arscalar}. Thus, we expect that the preconditioner $\ar$ will perform best when $\A$ is nearly symmetric, and we investigate this in \cref{sec:numerics}. However, irrespective of the degree of nonsymmetry of $\A$, \cref{thm:symeigs} shows that the eigenvalues of the preconditioned matrix are at least bounded away from the origin.

\subsection{The preconditioner $\af$}\label{sec:idealpretoepabs}
We saw in the previous section that $\ar$ is an effective preconditioner when the degree of nonsymmetry of $\A$ is not too large. For problems that are highly nonsymmetric, however, a different preconditioner may be more effective. Here, motivated by the success of absolute value preconditioning, we consider the preconditioner $\af = \A(|f|)$  instead. The following result describes the asymptotic eigenvalue distribution of $\af^{-1}\Y\A$. 
\begin{theorem}
\label{thm:absfeigs}
Assume that $f \in L^{\infty}([-\pi,\pi])$ with $0<\delta \le |f(\fvar)|$ for all $\fvar\in[-\pi,\pi]$. Then, if $\af = \A(|f|)$, 
\[(\af)^{-1}\Y\A(f) = \Y\A(\widetilde{f}) +  \En,\]
where $\widetilde{f} = f/|f|$ and  $\|\En\|_2 =o(n)$ as $n\rightarrow \infty$. Moreover, the eigenvalues of $\Y\A(\widetilde{f})$ lie in $[-1,1]$. 
\end{theorem}
\begin{proof}
{The conditions on $|f|$ guarantee that $\A(|f|)$ is invertible and that its eigenvalues (singular values) are bounded away from 0. Thus, by Proposition 5 in~\cite{DCMS16},  
\[\A(|f|)^{-1}\A(f) - \A(\widetilde f) = \widetilde{E}_n,\]
where $\|\widetilde{E}_n\|_2 =o(n)$ as $n\rightarrow \infty$. Since $\af$ is Hermitian and Toeplitz,  both $\af$ and its inverse are centrosymmetric. 
It follows that 
\[(\af)^{-1}(\Y\A(f)) = \Y\left((\af)^{-1}\A(f)\right)  = \Y\A(\widetilde f) + \En,\]
where $\En = \Y \widetilde{E}_n$ and $\|\En\|_2 = \|Y\widetilde{E}_n\|_2 = \|\widetilde{E}_n\|_2$. Hence $\|\En\|_2 = o(n)$ as $n\rightarrow \infty$. 
Since $\Y$ is unitary and $\Y\A(\widetilde f)$ is symmetric, the absolute values of the eigenvalues of $\Y\A(\widetilde{f})$ coincide with the singular values of $\A(\widetilde f)$, which in turn are bounded above by one~\cite{Wido89}. This proves the result.}
\end{proof}

A consequence of \cref{thm:absfeigs} is that the eigenvalues of $\A(|f|)^{-1}\Y\A$ lie in $[-1-\epsilon,1+\epsilon]$, where for large enough $n$ the parameter $\epsilon$ is small. Although eigenvalues may be close to the origin, most cluster at $-1$ and 1, in line with Theorem 3.4 in \cite{MaPe19}.

To conclude this section, we show how $\af$ can be approximated by circulant preconditioners. First, 
recall from \cref{sec:backgroundtoep} that $\Ce(f)$ is the preconditioner with eigenvalues  $\lambda_j = f(2\pi j/n)$, $j=0,\dotsc, n-1$. For large enough dimension $n$, we have that $\af = \Ce(|f|) + E_n + R_n$, where $E_n$ has small norm and $R_n$ has small rank~\cite[pp. 108--110]{GaSe17}, so that $\Ce(|f|)$ is a good approximation to $\af$ for large $n$. 

The matrix $C(|f|)$ can in turn be approximated by the Strang absolute value circulant preconditioner $|\CSt|$~\cite{CNP96,NgPo01,PeWa15}, where if $\CSt$ is the Strang circulant preconditioner~\cite{Stra86} for $\A$, with eigenvalues $\lambda_j$, $j = 1,\dotsc, n$, then the corresponding absolute value circulant preconditioner $|\CSt|$ has eigenvalues $|\lambda_j|$, $j = 1,\dotsc, n$. For this preconditioner, we obtain the following result. 

\begin{theorem}
\label{thm:strangabs}
Let $f:[-\pi,\pi]\rightarrow \mathbb{C}$ be in the Wiener class and let $\A = \A(f) \in \Rnn$. 
Then the Strang preconditioner $\CSt$, is such that 
$|\CSt| \rightarrow \Ce(|f|)$ as $n\rightarrow\infty$.
\end{theorem}
\begin{proof}
Assume that $n$, the dimension of $\A$, is $n = 2m+1$. (The idea can be extended to the case of even $n$, as in \cite[p. 37]{ChYe92}.) Then, $\CSt = \Ce(\D_m\star f)$, where 
\[(\D_m\star f)(\fvar) = \frac{1}{2\pi}\int_{-\pi}^\pi \D_m(\phi) f(\fvar-y)\, \dd \phi = \sum_{-m}^m a_k \ee^{\frac{2\pi\ii j k}{n}} \]
is the convolution of $f$ with the Dirichlet kernel $\D$ \cite{ChYe92}, and $a_k$ is as in \cref{eqn:fourier}.  

Since $|\CSt| = |\Ce(\D_m\star f)|$ and $\Ce(|f|)$ are both diagonalized by the Fourier matrix, they will be identical if all their eigenvalues, defined by \cref{eqn:circeigs}, match. 
The eigenvalues of $\Ce(\D_m\star f)$ are $(\D_m\ast f)(2\pi j/n)$, $j = 0,\dotsc, n-1$. Hence the $j$th eigenvalue of $|\Ce(\D_m\star f)|$ is 
\[\lambda_j(|\Ce(\D_m\star f)|) = \left((\D_m\star f)\left(\frac{2\pi j}{n}\right)\, \overline{(\D_m\star f)\left(\frac{2\pi j}{n}\right)}\right)^{\frac{1}{2}}.\] 
Since $f$ is in the Wiener class, $(\D_m\star f)(\fvar)$ converges absolutely, hence uniformly to $f(\fvar)$. Thus, 
\begin{align*}
\lim_{n\rightarrow\infty}\lambda_j(|\Ce(\D_m\star f)|) &= \lim_{n\rightarrow\infty}\left((\D_m\star f)\left(\frac{2\pi j}{n}\right)\, \overline{(\D_m\star f)\left(\frac{2\pi j}{n}\right)}\right)^{\frac{1}{2}} \\
&= \left(f\left(\frac{2\pi j}{n}\right) \, \overline{f\left(\frac{2\pi j}{n}\right)}\right)^{\frac{1}{2}} 
= \left|f\left(\frac{2\pi j}{n}\right)\right|
= \lambda_j(\Ce(|f|).
\end{align*}
Since the eigenvalues of $|\Ce(\D_m\star f)|$ approach those of $\Ce(|f|)$ as $n\rightarrow \infty$  we obtain the result. 
\end{proof}

\section{Multilevel Toeplitz problems}
\label{sec:idealpreml}
We now extend the results of  \cref{sec:idealpre} to multilevel Toeplitz matrices.

\subsection{The preconditioner $\ar$}
\label{sec:idealpretoepml}
The results of \cref{sec:idealpretoep} carry over straightforwardly to the multilevel case. They depend on the following generalization of \cref{lem:eigfunction}. The result essentially appeared in Theorem 2.4\footnote{Although the result is stated for $f,g$ nonnegative, the proof also holds for indefinite $f$.} in \cite{Serr94}.  
\begin{lemma}[{\cite{Serr94}}]
\label{lem:eigfunctionml}
Let $f,g\in L^1([-\pi,\pi]^p)$  with $g$ essentially positive. Let 
\[r :=\essinf_{\fvar \in [-\pi,\pi]^p} \frac{f(\fvar)}{g(\fvar)} \quad R:=\esssup_{\fvar \in [-\pi,\pi]^p} \frac{f(\fvar)}{g(\fvar)}.\]
Then the eigenvalues of $\Aml^{-1}(g)\Aml(f)$  lie in $(r,R)$ if $r<R$. If $r=R$ then $\A^{-1}(g)\A(f)=I_n$, where $I_n$ is the identity matrix of dimension $\pi(n) = n_1\cdots n_p$.  
\end{lemma}

With this result,  \cref{lem:IEdecomp,thm:symdecomp,thm:symeigs} carry over directly to the multilevel case. In particular, we have the following  characterization of the eigenvalues of $\ar^{-\frac{1}{2}}(\Y\A)\ar^{-\frac{1}{2}}$. 

\begin{theorem}
\label{thm:symeigsml}
Let $\ff\in L^1([-\pi,\pi]^p)$, and let $\ff = \ff_R + \ii \ff_I$, where $\ff_R$ and $\ff_I$ are real-valued functions with $\ffr$ essentially positive. Additionally, let $\Aml:=\Aml(\ff)\in\mathbb{R}^{\pi(n)\times \pi(n)}$ be the multilevel Toeplitz matrix associated with $\ff$ and let $\arml = \Aml(\ffr)  = (\Aml + \Aml^T)/2$. Then, the eigenvalues of $\arml^{-\frac{1}{2}}(\Yml\Aml)\arml^{-\frac{1}{2}} $ lie in $[-1-\epsilon,-1]\cup[1,1+\epsilon]$, where 
\begin{equation}
\label{eqn:arscalarml}
\epsilon <\esssup_{\theta \in [-\pi,\pi]^p} \left(\left|\frac{\ff_I(\theta)}{\ff_R(\theta)}\right|\right).
\end{equation}
\end{theorem}

\cref{thm:symeigsml} characterizes the eigenvalues of $\arml^{-\frac{1}{2}}(\Yml\Aml)\arml^{-\frac{1}{2}} $, which are bounded away from the origin. In turn, this allows us to bound the convergence rate of preconditioned MINRES in terms of easily-computed quantity in \cref{eqn:arscalarml}.

\subsection{The preconditioner $\af$}
We can also extend the results in \cref{sec:idealpretoepabs} to the multilevel case. However, for multilevel problems this preconditioner is more challenging to approximate. Matrix algebra, e.g., block circulant, preconditioners will generally result in iteration counts that increase as the dimension increases, as previously discussed. On the other hand, constructing effective banded Toeplitz, or efficient multilevel, algorithms is challenging since it is generally necessary to compute elements of $\af$. Nonetheless, we present the following result for completeness. It directly generalizes the result for Toeplitz matrices, so is presented without proof. 


\begin{theorem}
\label{thm:absfeigsml}
Let $f: L^{\infty}([-\pi,\pi]^p)$, with  $0<\delta < |f(\fvar)|$ for all $\fvar\in[-\pi,\pi]^p$. Then, if $\af = \Aml(|f|)$ is the multilevel Toeplitz matrix generated by $|f|$, 
\[(\af)^{-1}\Y\A(f) = \Yml\Aml(\tilde{f}) +  \Eml,\]
where $\tilde{f} = f/|f|$ and  $\|\Eml\|_2 =o(n)$ as $n\rightarrow \infty$. Moreover, the eigenvalues of $\Yml\Aml(\tilde{f})$ lie in $[-1,1]$. 
\end{theorem}

\section{Numerical experiments} \label{sec:numerics}
In this section we investigate the effectiveness of the preconditioners described above, and approximations to them, for the symmetrized system \cref{eqn:symlinsys}. We also compare the proposed approach to using nonsymmetric preconditioners for \cref{eqn:linsys} within preconditioned GMRES and LSQR~\cite{PaSa92}. All code is written in {\sc Matlab} (version 9.4.0) and is run on a quad-core, 62 GB RAM, Intel i7-6700 CPU with 3.20GHz\footnote{Code is available from \url{https://github.com/jpestana/fracdiff}}.  
We apply {\sc Matlab} versions of LSQR and MINRES, and a version of GMRES that performs right preconditioning. (Note that LSQR requires two matrix-vector products with the coefficient matrix and two preconditioner solves per iteration.) We take as our initial guess $x_0 = (1,1,\dotsc,1)^T/\sqrt{n}$, and we stop all methods when $\|r_k\|_2/\|r_0\|_2 < 10^{-8}$. When more than 200 iterations are required, we denote this by `---' in the tables.

When $\ar$ or $\af$ are too expensive to apply directly we use either a circulant or multigrid approximation. The multigrid preconditioner consists of a single V-cycle with damped Jacobi smoothing and Galerkin projections, namely linear or bilinear interpolation and restriction by full-weighting. The coarse matrices are also built by projection. The number of smoothing steps and the damping factor $\omega$ are stated below for each problem. The damping parameter is chosen by trial-and-error to minimize the number of iterations needed for small problems. 
When applying circulant preconditioners to~\eqref{eqn:symlinsys} we use the absolute value preconditioner in \cref{eqn:abscirc}
based on the Strang~\cite{Stra86}, optimal~\cite{Chan88} or superoptimal~\cite{Tyrt92} circulant preconditioner.

\begin{example}
\label{ex:1}
Our first example is from~\cite[Example 2]{HST05}, where numerical experiments indicated that $\ar$ is an effective preconditioner for the nonsymmetric system \cref{eqn:linsys} when GMRES is applied. The Toeplitz coefficient matrix $\A = \A(\ff)$ is formed from the generating function $\ff(\fvar)=(2-2 \cos(\fvar))(1+\ii \fvar)$. Since computing the Fourier coefficients for larger problems is time-consuming, smaller problems examined here. The right-hand side is a random vector (computed using the {\sc Matlab} function \texttt{randn}). 

The preconditioner $\ar:=\A(\ffr)$ is positive definite, since $\ffr(\fvar) = 2-2 \cos(\fvar)$ is essentially positive. Indeed, $\ar$ is the second-order finite difference matrix, namely the tridiagonal matrix with $2$ on the diagonal and $-1$ on the sub- and super-diagonals. Accordingly, $\ar$ can be applied directly with $O(n)$ cost. For comparison we also apply the optimal circulant preconditioner $\Ce$ and its absolute value counterpart $|\Ce|$. (The optimal circulant outperformed the Strang and superoptimal circulant preconditioners for this problem.)  The absolute value circulant approximates $\af$. 

\begin{table}[htbp]
{\footnotesize
\setlength{\tabcolsep}{.3em}
\caption{Iteration numbers and CPU times (in parentheses) for the optimal circulant preconditioner $\Ce$ and tridiagonal preconditioner $\ar$ for \cref{ex:1}.  }
\label{tab:ex1_ar}
\begin{center}
\begin{tabular}{r  *{3}{|rrrr}}
$n$ & \multicolumn{4}{c|}{GMRES} &  \multicolumn{4}{c|}{LSQR} &  \multicolumn{4}{c}{MINRES}\\ 
& \multicolumn{2}{c}{$\Ce$} & \multicolumn{2}{c}{$\ar$} & \multicolumn{2}{|c}{$\Ce$}  & \multicolumn{2}{c}{$\ar$} & \multicolumn{2}{|c}{$|\Ce|$} &  \multicolumn{2}{c}{$\ar$}\\
\hline
1023 & 37 & (0.058) & 67 & (0.067) & 105 & (0.15) & 62 & (0.05) & 82 & (0.057) & 68 & (0.029) \\
2047 & 48 & (0.13) & 68 & (0.12) & 186 & (0.4) & 67 & (0.081) & 111 & (0.12) & 70 & (0.046) \\
4095 & 62 & (0.25) & 69 & (0.22) & --- & --- & 73 & (0.13) & 170 & (0.21) & 71 & (0.065) \\
8191 & --- & --- & 72 & (0.51) & --- & --- & 78 & (0.2) & --- & --- & 72 & (0.13) \\
\end{tabular}
\end{center}
}
\end{table}

\cref{tab:ex1_ar} shows that  $\ar$ requires fewer iterations than $\Ce$ for MINRES and LSQR, and that MINRES with $\ar$ is the fastest method overall. The good performance of $\ar$ with MINRES can be explained by the clustered eigenvalues of $\ar^{-1}\A$. \cref{thm:symeigs} tells us that these eigenvalues lie in $[-1-\pi,-1]\cup [1,1+\pi]$, and  \cref{fig:ex1_ar} (b) shows that these bounds are tight. As discussed in~\cite{HST05}, the eigenvalues of $\ar^{-1}\A$ are also nicely clustered (see \cref{fig:ex1_ar} (a)), with real part $1$ and imaginary part in $[-\pi,\pi]$. Although we cannot rigorously link this  eigenvalue characterization to the rate of GMRES convergence, \cref{tab:ex1_ar} indicates that in this case $\ar$ is also a reasonable preconditioner for GMRES. 

\begin{figure}
\centering
\begin{subfigure}[b]{0.48\textwidth}
        \includegraphics[width=\textwidth,clip=true,trim = 3cm 9cm 4cm 9cm]{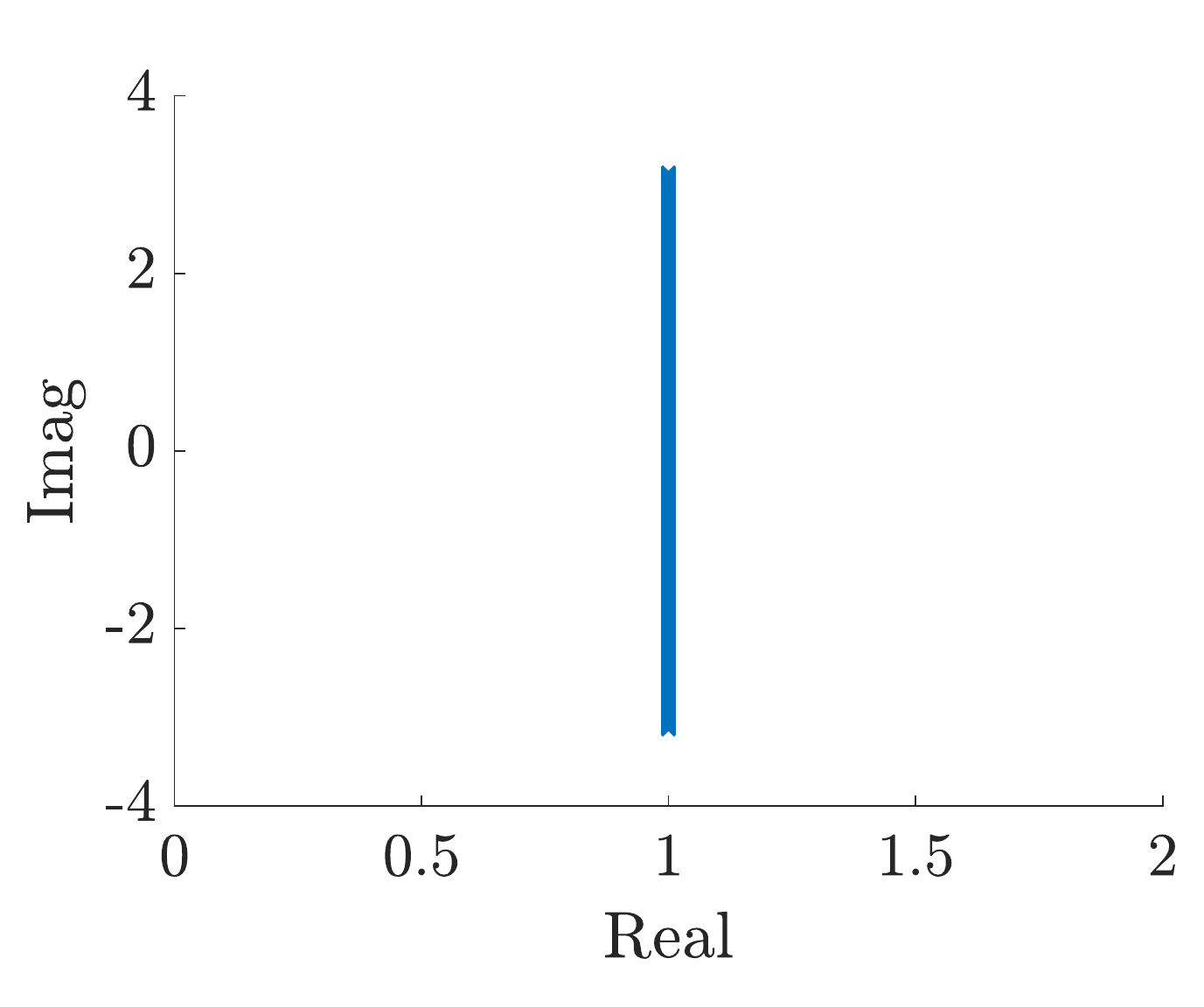}
        \caption{$\A$}
    \end{subfigure}
    \begin{subfigure}[b]{0.48\textwidth}
        \includegraphics[width=\textwidth,clip=true,trim = 3cm 9cm 4cm 9cm]{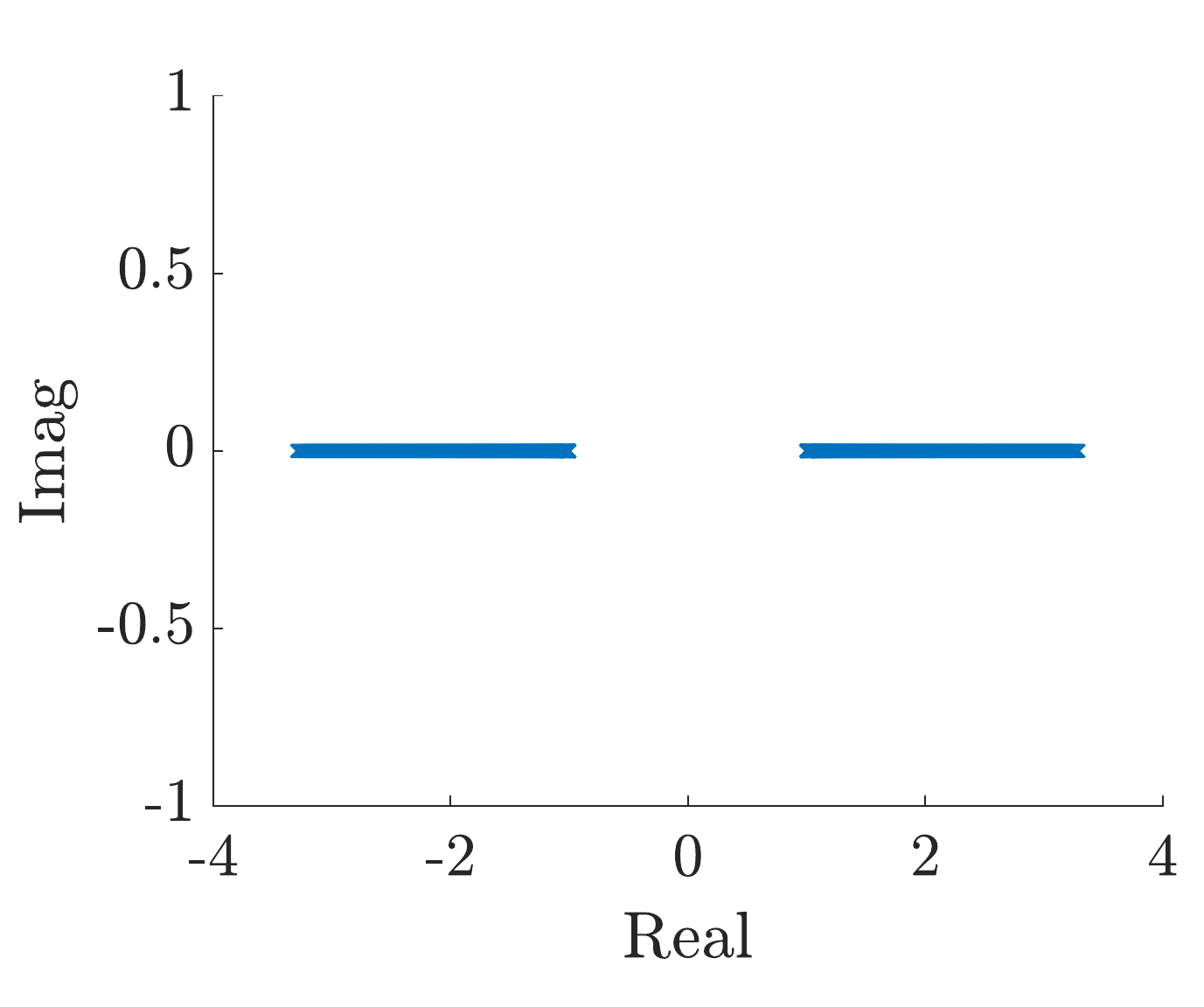}
        \caption{$\Y\A$}
    \end{subfigure}
\caption{Eigenvalues of $\ar^{-1}\A$ and $\ar^{-1}\Y\A$ for \cref{ex:1} with $n = 2047$. }
\label{fig:ex1_ar}
\end{figure}

We now consider $\af$, which is dense since $|f(\fvar)| = (2-2\cos(\fvar))\sqrt{1+\fvar^2}$. Accordingly, as well as applying $\af$ exactly---to confirm our theoretical results---we  approximate $\af$ via our V-cycle multigrid method with 2 pre- and 2 post-smoothing steps,  the coarsest grid of dimension 15, and $\omega = 0.1$ for GMRES, $\omega = 0.4$ for LSQR and $\omega = 0.5$ for MINRES. For LSQR, multigrid with $\af$ gave lower timings and iteration counts than multigrid with $\A$, and so was used instead. 

Iteration counts and CPU times (excluding the time to construct $\af$ but including the time to set up the multigrid preconditioner) are given in \cref{tab:ex1_af}. Both $\af$ and its multigrid approximation give lower iteration counts than $\ar$, with the multigrid method especially effective for MINRES applied to the symmetrized system. However, timings are higher than for $\ar$ since the $O(n\log(n))$ multigrid method is more expensive than the $O(n)$ solve with $\ar$. The eigenvalues of $\af^{-1}\Y\A$, when $n = 2047$, are as expected from \cref{thm:absfeigs} (see \cref{fig:ex1_af}), since all eigenvalues lie in $[-1,1]$. Indeed, most cluster at the endpoints of this interval. The eigenvalues of $\af^{-1}\A$ are also localized, but not as clustered, indicating that the spectrum of $\af^{-1}\A$ may differ significantly from that of $\af^{-1}\Y\A$. 

\begin{table}[htbp]
{\footnotesize
\setlength{\tabcolsep}{.3em}
\caption{Iteration numbers and CPU times (in parentheses) for the exact preconditioner $\af$ and its multigrid approximation $MG(\af)$ for \cref{ex:1}.  The second column shows the time needed to compute the elements of $\af$. }
\label{tab:ex1_af}
\begin{center}
\begin{tabular}{r r *{3}{|rrrr}}
$n$ & $\af$ time & \multicolumn{4}{c|}{GMRES} &  \multicolumn{4}{c|}{LSQR} &  \multicolumn{4}{c}{MINRES}\\ 
& & \multicolumn{2}{c}{$\A$} & \multicolumn{2}{c}{$MG(\A$)} & \multicolumn{2}{|c}{$\A$} & \multicolumn{2}{c}{$MG(\af)$}  & \multicolumn{2}{|c}{$\af$} &  \multicolumn{2}{c}{$MG(\af)$}\\
\hline
1023 & 6.4 & 1 & (0.06) & 39 & (0.15) & 1 & (0.068) & 33 & (0.11) & 11 & (0.11) & 24 & (0.044) \\
2047 &  21 & 1 & (0.28) & 41 & (0.28) & 1 & (0.33) & 37 & (0.24) & 11 & (0.6) & 24 & (0.082) \\
4095 &  73 & 1 & (1.9) & 39 & (0.36) & 1 & (1.9) & 39 & (0.27) & 12 & (3.6) & 25 & (0.096) \\
8191 & $2.5\times 10^2$ & 1 & (8.7) & 42 & (0.88) & 1 & (11) & 43 & (0.67) & 12 & ( 22) & 25 & (0.21) \\
\end{tabular}
\end{center}
}
\end{table}

\begin{figure}
\centering
\begin{subfigure}[b]{0.48\textwidth}
        \includegraphics[width=\textwidth,clip=true,trim = 3cm 9cm 4cm 9cm]{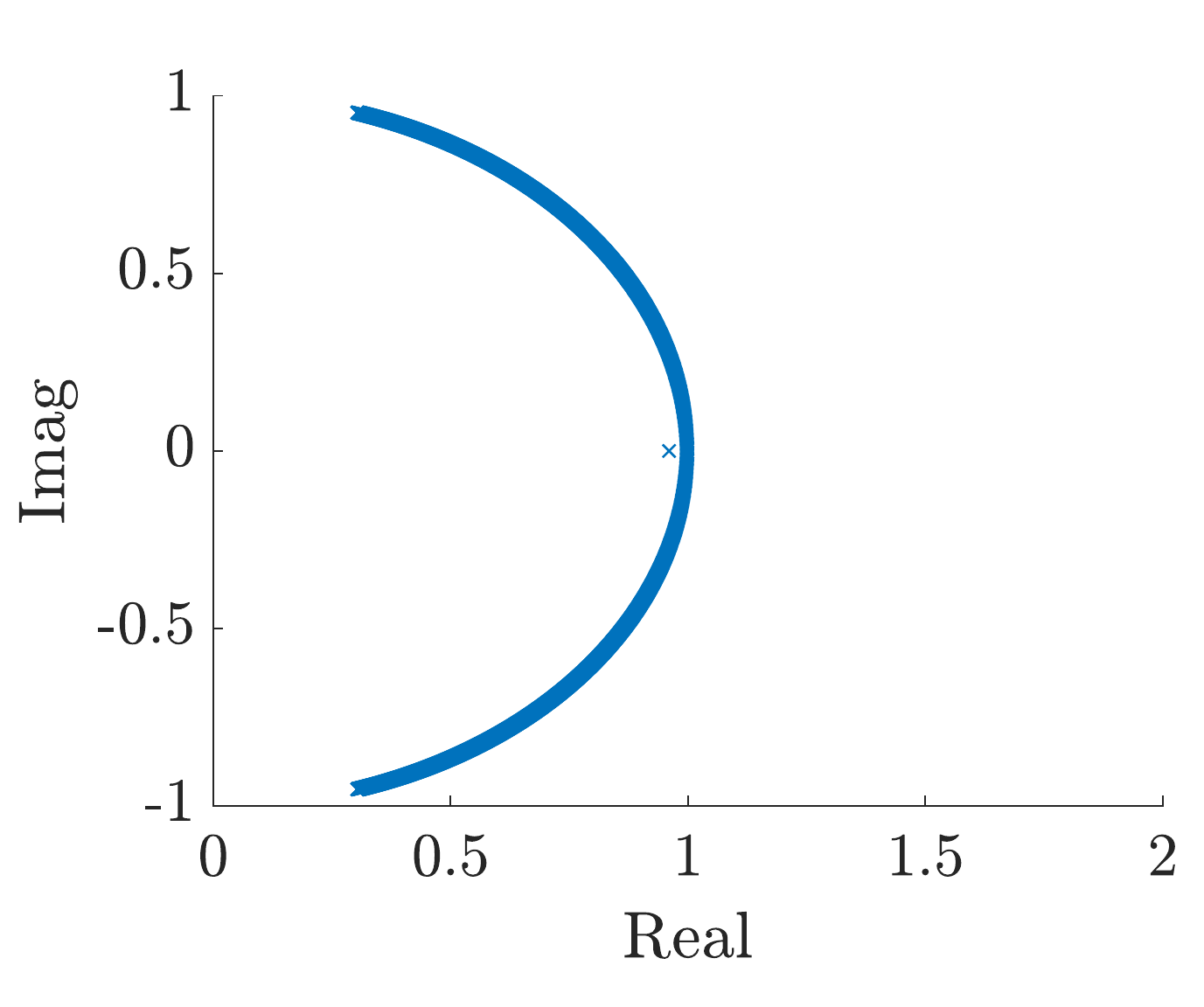}
        \caption{$\A$}
    \end{subfigure}
    \begin{subfigure}[b]{0.48\textwidth}
        \includegraphics[width=\textwidth,clip=true,trim = 3cm 9cm 4cm 9cm]{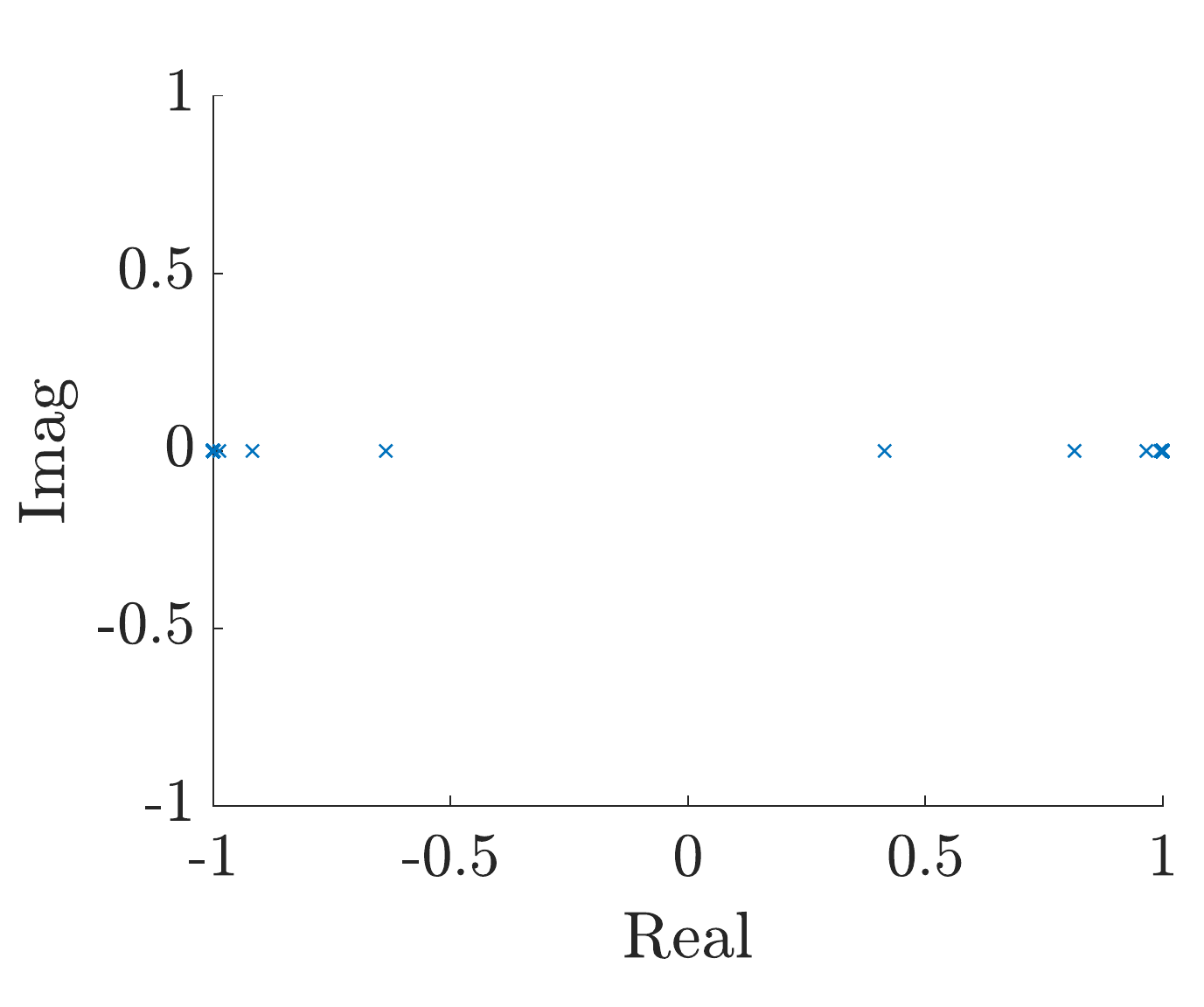}
        \caption{$\Y\A$}
    \end{subfigure}
\caption{Eigenvalues of $\af^{-1}\A$ and $\af^{-1}\Y\A$ for \cref{ex:1} with $n = 2047$. }
\label{fig:ex1_af}
\end{figure}

\end{example}

\begin{example}
\label{ex:2}
We now examine the linear system obtained by discretising a fractional diffusion problem from \cite{BSS16}, which we alter so as to make it nonsymmetric. The problem is to find $u(x,t)$ that satisfies 
\begin{equation}
\label{eqn:fracdiffeqn}
\frac{\partial u(x,t)}{\partial t}  = d_+ \frac{\partial_+^\alpha u(x,t)}{\partial x^\alpha} + d_-\frac{\partial^\alpha_- u(x,t)}{\partial x^\alpha} + f(x,t),\quad (x,t) \in (0,1) \times (0,1],
\end{equation}
where $\alpha \in (1,2)$, and $d_+$ and $d_-$ are nonnegative constants. We impose the absorbing boundary conditions $u(x\le 0,t) = u(x\ge 1,t) = 0$, $t \in [0,1]$, while 
$u(x,0) = 80\sin(20x)\cos(10x), \,  x  \in [0,1]. $
The Riemann-Liouville derivatives in \cref{eqn:fracdiffeqn} are 
\begin{align*}
 \frac{\partial_+^\alpha u(x,t)}{\partial x^\alpha} &= \frac{1}{\Gamma(n-\alpha)} \frac{\partial^n}{\partial x^n}\int_L^x \frac{u(\xi,t)}{(x-\xi)^{\alpha+1-n} }\dd\xi,\\
 \frac{\partial_-^\alpha u(x,t)}{\partial x^\alpha} &= \frac{(-1)^n}{\Gamma(n-\alpha)} \frac{\partial^n}{\partial x^n}\int_x^R \frac{u(\xi,t)}{(\xi-x)^{\alpha+1-n}} \dd\xi,
\end{align*}
where $n$ is the integer for which $n-1< \alpha \le n$. 

Discretising by the shifted Gr\"unwald-Letnikov method in space, and the backward Euler method in time~\cite{MeTa04,MeTa06}, gives the linear system 
\begin{equation}
\underbrace{(\nu I + d_+ L_\alpha + d_-L_\alpha^T)}_{A} u^m = \nu u^{m-1} + h^\alpha f^{m},\label{eqn:fracdiff}
\end{equation}
\begin{equation}
\label{eqn:fracdiffmat}
L _{\alpha}= -\begin{bmatrix}
 g_{\alpha,1}  &g_{\alpha,0} & & & \\
g_{\alpha,2}& g_{\alpha,1}  & g_{\alpha,0} & & \\
\vdots & \ddots & \ddots & \ddots & \\
g_{\alpha,n-1}& & \ddots& \ddots & g_{\alpha,0} \\
g_{\alpha,n}&g_{\alpha,n-1} &\dots & g_{\alpha,2}& g_{\alpha,1}\\
 \end{bmatrix}, 
 \end{equation}
 where $g_{\alpha,k} = (-1)^k \binom{\alpha}{k}$, $\nu =\frac{\tau}{h^\alpha}$ and $h = \frac{1}{n+1}$.
We set $\tau = 1/\lceil n^\alpha \rceil$, which makes $\nu$ constant, so that all the theory of \cref{sec:idealpre} can be directly applied, but comparable results are obtained  $\tau = 1/n$. 
Stated CPU times and iteration counts in this example are for the first time step. (Iteration counts and timings decrease at later time steps.) CPU times include the preconditioner setup time and solve time. 

Entries of $A$ in \cref{eqn:fracdiff} are generated by~\cite{DMS16}  
\begin{equation*}
\varphi(\theta) = \nu + d_+ f_\alpha(\fvar) + d_- f_\alpha(-\fvar), \qquad f_\alpha(\fvar) = -\ee^{-\ii\fvar}\left(1-\ee^{\ii\theta} \right)^\alpha.
\end{equation*}
The real part of $\varphi$ is essentially positive, so $\ar=(\A+\A^T)/2$ is positive definite. However, since $\ar$ is dense we approximate it by our V-cycle multigrid method (analysed in~\cite{PaSu12}) with the coarsest grid of dimension 127, 2 pre- and 2 post-smoothing steps, and $\omega = 0.7$ for all Krylov solvers. 
The matrix $\af$ is also dense and positive definite, and we approximate it using two different approaches. The first is the absolute value Strang preconditioner discussed at the end of \cref{sec:idealpretoepabs}. The second is multigrid (with the same parameters as for $\ar$, except that we use 1 pre- and post-smoothing step) applied to a banded Toeplitz approximation of $\af$. Specifically, if $r$ and $c$ are the first row and column of $\af$, when $\alpha= 1.25$ we compute the first 50 elements in $r$ and $c$, and when $\alpha>1.25$ we take the first  $\lceil \beta (1.1)^{\log_2(n+1)}\rceil$ elements in $r$ and $c$, where $\beta = 40$ when $\alpha = 1.5$ and $\beta = 100$ when $\alpha = 1.75$. This balances the time to compute these coefficients, and the resulting MINRES iteration count. 

We see from \cref{tab:ex2_ar_alpha} that our approximations to $\ar$ and $\af$  are robust with respect to $n$, but both require slightly more iterations for larger $\alpha$. The multigrid preconditioner for $\ar$ requires fewer iterations than the circulant, but the latter results in a lower CPU time because the preconditioner application is cheap, and indeed the absolute value preconditioner with MINRES is the fastest method overall. Of the multigrid methods, the approximation to $\ar$ with MINRES is fastest for $\alpha \le 1.5$, while the multigrid approximation of $\A$ with GMRES is slightly faster for large $\alpha$.

\begin{table}[htbp]
{\footnotesize
\setlength{\tabcolsep}{.3em}
\caption{Iteration numbers and CPU times (in parentheses) for the Strang circulant $\Ce$, absolute value Strang circulant $|\Ce|$ and multigrid preconditioners when $d_+ = 0.5$ and $d_-=1$ for \cref{ex:2}.  }
\label{tab:ex2_ar_alpha}
\begin{center}
\begin{tabular}{r r  *{2}{|rrrr} |rrrrrr}
$\alpha$ & $n$ & \multicolumn{4}{c|}{GMRES} &  \multicolumn{4}{c|}{LSQR} &  \multicolumn{6}{c}{MINRES}\\ 
& & \multicolumn{2}{c}{$\Ce$} & \multicolumn{2}{c|}{$MG(\A)$} &  \multicolumn{2}{c}{$\Ce$} & \multicolumn{2}{c|}{$MG(\A)$}   &  \multicolumn{2}{c}{$|\Ce|$}  &  \multicolumn{2}{c}{$MG(\af)$} & \multicolumn{2}{c}{$MG(\ar)$} \\
\hline
\multirow{5}{*}{1.25} & 1023 & 5 & (0.01) & 4 & (0.016)& 6 & (0.011) & 6 & (0.02)& 10 & (0.0084) & 12 & (0.18)& 8 & (0.014)\\
& 4095 & 6 & (0.017) & 4 & (0.045)& 6 & (0.016) & 6 & (0.053)& 10 & (0.013) & 12 & (0.18)& 8 & (0.043)\\
& 16383 & 6 & (0.065) & 4 & (0.17)& 6 & (0.066) & 7 & (0.22)& 10 & (0.054) & 13 & (0.32)& 8 & (0.17)\\
& 65535 & 6 & (0.25) & 4 & (0.66)& 6 & (0.25) & 7 & (0.76)& 9 & (0.19) & 13 & (0.82)& 8 & (0.6)\\
& 262143 & 6 & (0.99) & 4 & (4.4)& 6 & (  1) & 7 & (  5)& 9 & (0.72) & 13 & (4.5)& 8 & (4.3)\\
\hline
\multirow{5}{*}{1.5} & 1023 & 6 & (0.0062) & 4 & (0.021)& 6 & (0.0062) & 7 & (0.025)& 10 & (0.0048) & 13 & (0.37)& 8 & (0.013)\\
& 4095 & 6 & (0.018) & 4 & (0.046)& 6 & (0.018) & 7 & (0.061)& 10 & (0.015) & 13 & (0.5)& 8 & (0.044)\\
& 16383 & 6 & (0.062) & 4 & (0.17)& 6 & (0.067) & 7 & (0.21)& 9 & (0.05) & 13 & (0.76)& 9 & (0.19)\\
& 65535 & 6 & (0.24) & 5 & (0.7)& 7 & (0.28) & 8 & (0.8)& 9 & (0.19) & 13 & (1.4)& 9 & (0.66)\\
& 262143 & 6 & (0.93) & 5 & (5.2)& 7 & (1.1) & 8 & (5.7)& 9 & (0.72) & 15 & (6.1)& 9 & (4.7)\\
\hline
\multirow{5}{*}{1.75} & 1023 & 6 & (0.0085) & 5 & (0.043)& 7 & (0.0088) & 7 & (0.021)& 9 & (0.0062) & 13 & (1.6)& 9 & (0.014)\\
& 4095 & 6 & (0.015) & 5 & (0.046)& 7 & (0.015) & 8 & (0.058)& 9 & (0.01) & 15 & (2.2)& 9 & (0.036)\\
& 16383 & 6 & (0.062) & 5 & (0.2)& 7 & (0.075) & 8 & (0.24)& 9 & (0.049) & 15 & (3.2)& 10 & (0.21)\\
& 65535 & 6 & (0.24) & 5 & (0.71)& 7 & (0.28) & 8 & (0.81)& 9 & (0.19) & 15 & (4.8)& 11 & (0.75)\\
& 262143 & 6 & (0.9) & 5 & (5.2)& 7 & (1.1) & 9 & (6.3)& 9 & (0.72) & 16 & ( 11)& 11 & (5.7)\\

\end{tabular}
\end{center}
}
\end{table}

In \cref{tab:ex2_ar_diff} we investigate the effect of $d_+$ and $d_-$, i.e., of nonsymmetry, on the preconditioners. The results are unchanged when $d_+$ and $d_-$ are swapped, so we tabulate results for $d_+\le d_-$ only. As expected, our approximation to $\ar$ is best suited to problems for which $d_+$ and $d_-$ do not differ too much. 
The hardest problem for $\ar$ is when $d_+=0$, since in this case $\A$ is a Hessenberg matrix, and hence highly nonsymmetric. However, even here the iteration numbers are fairly low, since the eigenvalues are bounded away from the origin independently of $n$. The circulant and multigrid preconditioners based on $\af$ are not greatly affected by altering $d_+$ and $d_-$.

\begin{table}[htbp]
{\footnotesize
\setlength{\tabcolsep}{.3em}
\caption{Iteration numbers and CPU times (in parentheses) for the Strang circulant $\Ce$, absolute value Strang circulant $|\Ce|$ and multigrid preconditioners when $\alpha = 1.5$ for \cref{ex:2}.  }
\label{tab:ex2_ar_diff}
\begin{center}
\begin{tabular}{r r  *{2}{|rrrr} |rrrrrr}
$(d_+,d_-)$ & $n$ & \multicolumn{4}{c|}{GMRES} &  \multicolumn{4}{c|}{LSQR} &  \multicolumn{6}{c}{MINRES}\\ 
& & \multicolumn{2}{c}{$\Ce$} & \multicolumn{2}{c|}{$MG(\A)$} &  \multicolumn{2}{c}{$\Ce$} & \multicolumn{2}{c|}{$MG(\A)$}   &  \multicolumn{2}{c}{$|\Ce|$}  &  \multicolumn{2}{c}{$MG(\af)$} & \multicolumn{2}{c}{$MG(\ar)$} \\
\hline
\multirow{4}{*}{(0,3)}& 4095 & 5 & (0.031) & 5 & (0.053)& 7 & (0.02) & 5 & (0.059)& 10 & (0.016) & 10 & (0.47)& 13 & (0.049)\\
& 16383 & 4 & (0.044) & 5 & (0.21)& 7 & (0.078) & 5 & (0.23)& 10 & (0.058) & 10 & (0.83)& 13 & (0.27)\\
& 65535 & 4 & (0.18) & 6 & (0.84)& 7 & (0.29) & 6 & (0.93)& 10 & (0.22) & 10 & (1.6)& 14 & (0.97)\\
& 262143 & 4 & (0.75) & 6 & (6.2)& 7 & (1.2) & 6 & (6.7)& 11 & (0.92) & 11 & (6.8)& 14 & (7.1)\\
\hline
\multirow{4}{*}{(1,3)}& 4095 & 7 & (0.015) & 5 & (0.042)& 7 & (0.014) & 5 & (0.045)& 10 & (0.013) & 10 & (0.4)& 9 & (0.037)\\
& 16383 & 7 & (0.072) & 5 & (0.21)& 7 & (0.078) & 5 & (0.23)& 11 & (0.06) & 10 & (0.81)& 10 & (0.21)\\
& 65535 & 7 & (0.29) & 5 & (0.71)& 8 & (0.33) & 5 & (0.77)& 11 & (0.24) & 11 & (1.6)& 10 & (0.7)\\
& 262143 & 7 & (1.1) & 6 & (6.2)& 8 & (1.3) & 6 & (6.6)& 11 & (0.93) & 11 & (6.6)& 10 & (5.3)\\
\hline
\multirow{4}{*}{(1,1)}& 4095 & 6 & (0.013) & 4 & (0.031)& 6 & (0.013) & 5 & (0.041)& 10 & (0.01) & 9 & (0.39)& 9 & (0.034)\\
& 16383 & 6 & (0.064) & 4 & (0.17)& 6 & (0.068) & 5 & (0.23)& 10 & (0.058) & 9 & (0.79)& 9 & (0.19)\\
& 65535 & 6 & (0.25) & 4 & (0.57)& 6 & (0.26) & 5 & (0.77)& 9 & (0.19) & 9 & (1.4)& 9 & (0.64)\\
& 262143 & 6 & (0.93) & 5 & (5.2)& 7 & (1.2) & 5 & (5.8)& 9 & (0.79) & 9 & (5.9)& 9 & (4.9)\\
\end{tabular}
\end{center}
}
\end{table}

The low iteration numbers and mesh-size independent results for $\ar$ in \cref{tab:ex2_ar_diff} are explained by \cref{thm:symeigs} and the relatively small upper bound \cref{eqn:arscalar}, which describes how far eigenvalues of $\ar^{-1}\Y\A$ can deviate from 1 in magnitude. This bound is 0 when $d_+=d_-$, or when $\alpha= 2$, since then $\A$ is symmetric. However,  \cref{tab:ex2_al_maxeig} shows that even when $\A$ is nonsymmetric the bound is quite small. Additionally, it does not change when the values of $d_+$ and $d_-$ are swapped. 

\begin{table}[htbp]
{\footnotesize
\setlength{\tabcolsep}{.3em}
\caption{Upper bound  in \cref{eqn:arscalar} for \cref{ex:2}.  }
\label{tab:ex2_al_maxeig}
\begin{center}
\begin{tabular}{r |  *{4}{r}}
$\alpha$ & \multicolumn{4}{c}{$(d_+,d_-)$}\\
 & (0,3) & (1,3) & (0.5,1) & (1,1) \\
\hline
1 & 1.13 & 0.67 & 0.25 & 0.00\\
1.25 & 0.70 & 0.39 & 0.17 & 0.00\\
1.5 & 0.42 & 0.23 & 0.11 & 0.00\\
1.75 & 0.20 & 0.11 & 0.05 & 0.00\\
\end{tabular} 
\end{center}
}
\end{table}

\end{example}

\begin{example}
\label{ex:3}
We now solve a two-level Toeplitz problem that also arises from fractional diffusion and is based on the symmetric problem in~\cite{BSS16}. We seek $u(x,y,t)$ in the domain  $\Omega = (0,1)^2 \times (0,1]$ that satisfies
\begin{align*}
\label{eqn:fracdiffeqn2d}
\frac{\partial u(x,y,t)}{\partial t}  =& d_+ \frac{\partial_+^\alpha u(x,y,t)}{\partial x^\alpha} + d_-\frac{\partial^\alpha_- u(x,y,t)}{\partial x^\alpha}\\& + e_+ \frac{\partial_+^\beta u(x,y,t)}{{\partial y^\beta}} + e_-\frac{\partial^\beta_- u(x,y,t)}{{\partial y^\beta}} + f(x,y,t),
\end{align*}
where $\alpha,\beta \in (1,2)$, and $d_+$, $d_-$, $e_+$ and $e_-$ are nonnegative constants. We impose absorbing boundary conditions, and the initial condition is 
$u(x,0) =100\sin(10x)\cos(y) + \sin(10t)xy$.

We again discretize by the shifted Gr\"unwald-Letnikov method in space, and the backward Euler method in time~\cite{MeTa04,MeTa06}, which leads to the following linear system:
\begin{equation}
\label{eqn:fracdiff2d}
\underbrace{(I_{n_xn_y} - I_{n_y}\otimes L_x - L_y\otimes I_{n_x})}_{\A} u^m =  u^{m-1} + \tau f^{m}.
\end{equation}
Here  $n_x$ and $n_y$ are the number of spatial degrees of freedom in the $x$ and $y$ directions, respectively; we choose $n_x=n_y=n$. Also, 
\[L_x = \frac{\tau}{h^\alpha_x}(d_+ L_{\alpha} + d_- L_{\alpha}^T),  \qquad L_y = \frac{\tau}{{h^\beta_y}}(e_+ L_{\beta} + e_- L_{\beta}^T),\] 
where $L_{\alpha}$
 is given by \cref{eqn:fracdiffmat}, and $h_x=1/(n_x+1)$ and $h_y=1/(n_y+1)$ are the mesh widths in the $x$ and $y$ directions. Unless $\alpha=\beta$,  $\tau/h_x^\alpha$ and $\tau/h_y^\beta$ cannot both be independent of $n$; we choose $\tau = 1/\lceil n_x^\alpha\rceil$. Note that the theory for $\ar$ still applies in this case. Stated CPU times and iteration counts are again for the first time step.

It is too costly to approximate $\af$ by a banded Toeplitz matrix or a multigrid method, simply because it is expensive to obtain the Fourier coefficients of $|f|$, and so we present results for a multigrid approximation to $\ar$ only below. 
We also apply the nonsymmetric block circulant $\Ce = I_{n_xn_y} - I_{n_y}\otimes C_x - C_y\otimes I_{n_x}$ preconditioner, and symmetric positive definite block circulant $|\Ce| = I_{n_xn_y} + I_{n_y}\otimes |C_x| + |C_y|\otimes I_{n_x}$, where $C_x$ and $C_y$ are Strang circulant approximations to $L_x$ and $L_y$, respectively. Our multigrid method comprises 4 pre- and 4 post-smoothing steps, and a damping parameter of 0.9. The coarsest grid has $n_x=n_y = 7$. 

The results in \cref{tab:ex3} show that the multigrid approximation of $\ar$ gives mesh-size independent iteration counts, and that MINRES with this preconditioner is the fastest method for larger problems. For the block circulant preconditioners we see different behaviour depending on whether $\alpha>\beta$. Specifically, when  $\alpha>\beta$, $\tau/h_y^\beta\rightarrow 0$ as $n\rightarrow \infty$, which makes this problem easier to solve in some sense. On the other hand, when $\alpha < \beta$ the problems become harder to solve as $n$ increases, and the block circulant with LSQR and MINRES suffer from growing iteration counts. 

\begin{table}[htbp]
{\footnotesize
\setlength{\tabcolsep}{.3em}
\caption{Iteration numbers and CPU times (in parentheses) for the  circulant preconditioners $\Ce$ and $|\Ce|$, and multigrid preconditioners when $d_+ = 2$, $d_-=0.5$, $e_+ = 0.3$ and $e_-=1$ for \cref{ex:3}.  }
\label{tab:ex3}
\begin{center}
\begin{tabular}{r r  *{3}{|rrrr} }
$(\alpha,\beta)$ & $n^2$ & \multicolumn{4}{c|}{GMRES} &  \multicolumn{4}{c|}{LSQR} &  \multicolumn{4}{c}{MINRES}\\ 
& & \multicolumn{2}{c}{$\Ce$} & \multicolumn{2}{c|}{$MG(\A)$} &  \multicolumn{2}{c}{$\Ce$} & \multicolumn{2}{c|}{$MG(\A)$}   &  \multicolumn{2}{c}{$|\Ce|$}  & \multicolumn{2}{c}{$MG(\ar)$} \\
\hline
\multirow{3}{*}{(1.5,1.25)} & 961 & 16 & (0.032) & 5 & (0.011)& 23 & (0.033)& 5 & (0.014) & 42 & (0.028) & 12 & (0.013)\\
& 16129 & 15 & (0.12) & 5 & (0.058)& 21 & (0.11)& 6 & (0.07) & 39 & (0.12) & 12 & (0.07)\\
& 261121 & 14 & (1.5) & 5 & (1.1)& 18 & (1.4)& 6 & (1.3) & 34 & (1.5) & 12 & (1.0)\\
\hline
\multirow{3}{*}{(1.5,1.75)} & 961 & 21 & (0.029) & 4 & (0.0086)& 28 & (0.038)& 4 & (0.0099) & 43 & (0.027) & 10 & (0.01)\\
& 16129 & 21 & (0.16) & 4 & (0.051)& 35 & (0.2)& 5 & (0.065) & 57 & (0.19) & 10 & (0.049)\\
& 261121 & 20 & (2.1) & 5 & (1.2)& 40 & (3.1)& 5 & (1.0) & 67 & (2.8) & 12 & (0.97)
\end{tabular}
\end{center}
}
\end{table}

\end{example}

\section{Conclusions}
\label{sec:conc}
In this paper we  presented two novel ideal preconditioners for (multilevel) Toeplitz matrices by considering the generating function $f$. The first, $\ar$ is formed using the real part of $f$. While it works best when the (multilevel) Toeplitz matrix is close to symmetric, it is reasonably robust with respect to the degree of nonsymmetry. This performance is likely attributable to the eigenvalue distribution, which remains bounded away from the origin. Our second preconditioner, $\af$, is based on $|f|$. Its performance is less affected by nonsymmetry, but it is more challenging to construct efficient approximations to $\af$ in the multilevel case. 

Our numerical results not only illustrate the effectiveness of the preconditioners, they highlight the value of  symmetrization, which enables us to compute bounds on convergence rates that depend only on the scalar function $f$. Additionally, the combination of symmetrization and preconditioned MINRES can be more computationally efficient than applying GMRES or LSQR to these problems. 

\section*{Acknowledgments}
The author would like to thank Mariarosa Mazza and Stefano Serra Capizzano for their careful reading of an earlier version of this manuscript and helpful discussions, and the anonymous referees for their suggestions. All data underpinning this publication are openly available from Zenodo at \url{http://doi.org/10.5281/zenodo.1327565}.

\bibliographystyle{siamplain}
\bibliography{WP1}      
\end{document}